\documentclass{article}
\usepackage{geometry}
\usepackage{amsmath, amsthm, amsfonts, amssymb, enumitem, esint}
\usepackage[hidelinks]{hyperref}

\numberwithin{equation}{section}

\newtheorem{theorem}{Theorem}[section]
\newtheorem{lemma}[theorem]{Lemma}
\newtheorem{proposition}[theorem]{Proposition}

\theoremstyle{definition}

\theoremstyle{remark}
\newtheorem{remark}{Remark}[theorem]

\DeclareMathOperator{\diam}{diam}
\DeclareMathOperator{\dist}{dist}
\DeclareMathOperator{\e}{e}

\DeclareMathOperator{\Imm}{Imm}
\DeclareMathOperator{\Ort}{O}
\DeclareMathOperator*{\osc}{osc}
\DeclareMathOperator{\rank}{rank}
\DeclareMathOperator{\SO}{SO}
\DeclareMathOperator{\lspan}{span}

\DeclareMathOperator{\vol}{vol}

\title{\textbf{Asymptotic rigidity of codimension-1 isometric immersions via quantitative estimates}}
\author{\textbf{Mert Baştuğ}}
\date{}

\begin{document}

\maketitle

\begin{abstract}
    We offer an alternative approach to the asymptotic rigidity of codimension-1 isometric immersions via quantitative rigidity estimates. We show that an immersion between compact manifolds $M$ and $N$ of dimensions $d$ and $d + 1$, respectively, with small stretching plus bending energy is close to an isometric immersion. In this way, we recover the results of Alpern, Kupferman, and Maor \cite{AKM, AKM2}. In contrast to their intrinsic approach, we reduce the problem to the equidimensional Euclidean setting and apply the Friesecke–James–Müller \cite{FJM} rigidity estimate to obtain quantitative results. This yields an elementary proof based on Euclidean techniques. The rigidity estimates are of independent interest.
\end{abstract}

\section{Introduction}

A classical theorem due to Liouville says that a map $u \in C^1(U; \mathbb{R}^d)$, $U \subset \mathbb{R}^d$, satisfying $Du(x) \in \SO(d)$ for all $x \in U$ must have a constant derivative if $U$ is connected, that is, $Du \equiv R \in \SO(d)$. Liouville's theorem is the starting point for a large body of research concerning the rigidity of isometric maps. Such theorems play a crucial role in the theory of elasticity and differential geometry. An important generalization of Liouville's theorem is the celebrated Friesecke-James-Müller quantitative rigidity estimate \cite{FJM} (see also \cite[Section 2.4]{CS}). Let $p \in (1, \infty)$. It states that for all $u \in W^{1, p}(U; \mathbb{R}^d)$ defined on a Lipschitz domain $U \subset \mathbb{R}^d$, there exists $R \in \SO(d)$ such that
\begin{equation} \label{eq:fjm}
    \|Du - R\|_{L^p} \le C \|\dist(Du, \SO(d))\|_{L^p},
\end{equation}
where $C$ depends only on $U$ and $p$. The estimate \eqref{eq:fjm} was a milestone in the rigorous derivation of lower-dimensional models such as plate and shell theories \cite{FJMM, FJM, FJM2, MS}.

Questions of rigidity naturally lend themselves to generalization in the Riemannian setting, since a map $u \in C^1(U; \mathbb{R}^d)$ satisfying $Du(x) \in \SO(d)$ for all $x \in U$ is an isometric embedding of $U$ in $\mathbb{R}^d$. Such generalizations are not only interesting from a mathematical viewpoint but also appear in non-Euclidean theory of elasticity (see \cite[Section 1.3]{KMS} for some applications). Let $(M, g)$ and $(N, h)$ be compact, oriented, $d$-dimensional Riemannian manifolds. Kupferman, Maor and Schachar \cite{KMS} proved the asymptotic rigidity of isometric immersions from $M$ to $N$: If $u_k \in W^{1, p}(M; N)$ and
\[
\int_M \dist_{g, h}^p(d(u_k)_x, \SO((T_xM, g_x), (T_{u_k(x)}N, h_{u_k(x)})) \, d\vol_g(x) \to 0
\]
(see Section \ref{sec:notation} for an explanation of the notation), then, up to a subsequence, $(u_k)$ converges in $W^{1, p}$ to a smooth isometric immersion. A quantitative version was obtained recently for self-maps by Conti, Dolzmann and Müller \cite{CDM}. Essentially, they prove that for every $u \in W^{1, p}(M; M)$, there exists an orientation preserving isometry $\phi$ such that the $W^{1, p}$-distance between $u$ and $\phi$ is bounded by the ``elastic energy"
\[
\int_M \dist_g^p(du_x, \SO(T_xM, g_x)) \, d\vol_g(x).
\]
We refer the reader to \cite[Figure 1]{AKM2} for a summary of the literature. 

The rigidity property enjoyed by isometric embeddings between equidimensional Euclidean spaces fails when the codomain has higher dimension, as illustrated by unit-speed curves in $\mathbb{R}^2$. The reason is that the derivative reflects local stretching properties but does not capture the degree of bending in the ambient space. The latter information is encoded by the second fundamental form or the shape operator. In the following, we consider $M$ and $N$ to be $d$ and $(d + 1)$-dimensional, respectively. Given an immersion $u : M \rightarrow N$ and $x \in M$, let $\nu_u(x) \in T_{u(x)}N$ be the unit normal to $du_x(T_xM)$ that is consistent with the orientation induced by $du_x$. We denote by $S_u : TM \rightarrow TM$ the shape operator induced by $u$ (see Section \ref{sec:sobolev_immersions} for the definition). We consider an elastic energy that measures the bending as well as the stretching energies. We set
\begin{equation*}
    E(u) := \int_M \dist_{g, h}^p(du_x, \Ort((T_xM, g_x), (T_{u(x)}N, h_{u(x)}))) \, d\vol_g(x) + \int_M |du_x \circ (S_u(x) - S(x))|_{g, h}^p \, d\vol_g(x).
\end{equation*}
Here $S : TM \rightarrow TM$ is a reference shape operator on $M$. The natural class of functions on which the energy is well-defined is the space of Sobolev immersions given by
\[
\Imm_p(M; N) := \{u \in W^{1, p}(M; N) : \rank du_x = d \text{ for a.e. } x \in M, \, \nu_u \in W^{1, p}(M; TN)\}.
\]
Alpern, Kupferman, and Maor proved in \cite[Theorem 1.1]{AKM2} that for compact manifolds $M$ and $N$, any sequence $(u_k) \in \Imm_p(M; N)$ with $E(u_k) \to 0$ admits a subsequence converging in $W^{1, p}$ to a smooth isometric immersion $u \in \Imm_p(M; N)$. Moreover, the associated unit normals $(\nu_{u_k})$ converge to $\nu_u$ in $W^{1, p}$, and $S = S_u$. This result builds on their earlier work \cite{AKM}, where the target manifold $N$ was additionally assumed to have constant sectional curvature.

The purpose of this work is to provide an alternative proof of the asymptotic rigidity theorem for codimension-1 isometric immersions based on local quantitative rigidity estimates. In comparison to the Young measure approach used in \cite{AKM2}, we rely only on the FJM estimate \eqref{eq:fjm} and elementary methods. We briefly describe the main idea behind our local rigidity estimate in the special case $N = \mathbb{R}^{d + 1}$. Let $u$ be a Sobolev immersion with small elastic energy $E(u)$. Then the small bending energy implies that the normal field $\nu_u$ varies little on sufficiently small scales. As a consequence, the image of $u$ is locally close to a $d$-dimensional affine subspace. Since $M$ is locally Euclidean, this allows us to view $u$, after restriction to small regions, as an approximately equidimensional map. The Euclidean rigidity estimate then shows that $u$ is locally well approximated by a rotation. When $N$ is an arbitrary compact manifold, we embed it isometrically in a Euclidean space. See Theorem \ref{thm:codimension1_compact_rigidity} for the precise statement of the result. Once we have the local rigidity estimate, we can recover the asymptotic rigidity result by a compactness argument similar to \cite[Theorem 4.1]{FJM}.

In a companion paper \cite{B2}, we extend the result by Alpern, Kupferman and Maor to complete manifolds $N$, dropping the compactness assumption. In this setting, it becomes more natural to work with the intrinsic definition of manifold-valued Sobolev maps, which requires technical refinements of the previous argument. The approach through embedding the target manifold isometrically into a Euclidean space offers a simplification and yields an elementary proof of Alpern, Kupferman and Maor. The rigidity estimate is also of independent interest.

The paper is organized as follows. Section \ref{sec:preliminaries} introduces some basic notation and the main class of functions we work with, the Sobolev immersions. In Section \ref{sec:local_rigidity}, we prove the local quantitative rigidity estimate stated in Theorem \ref{thm:codimension1_compact_rigidity}. In the final section of the paper, we recover the result by Alpern, Kupferman and Maor, given in Theorem \ref{thm:asymptotic_rigidity}.

\section{Preliminaries} \label{sec:preliminaries}

In this section, we define the main space of functions we work with. The first subsection introduces notation that we use frequently in the paper.

\subsection{Notation} \label{sec:notation}

Let $V$ be a vector space endowed with an inner product or, equivalently, a constant metric $g_0$. If $v, v' \in V$, then we simply write $(v, v')_{g_0}$ for $g_0(v, v')$ and $|v|_{g_0}$ for $\sqrt{g_0(v, v)}$. Let $W$ be another vector space with a constant metric $h_0$. The set of linear maps from $V$ to $W$ is denoted by $\mathcal{L}(V, W)$. We denote the subset of isometries in $\mathcal{L}(V, W)$ by $\Ort((V, g_0), (W, h_0))$, and the set of orientation preserving isometries is denoted by $\SO((V, g_0), (W, h_0))$. Let $(v_1, \dots, v_d)$ be an orthonormal basis of $V$. We define the Frobenius norm of $T \in \mathcal{L}(V, W)$ by
\[
|T|_{g_0, h_0} := \left(\sum_{i = 1}^d |Tv_i|_{h_0}^2\right)^\frac{1}{2}.
\]
If $(v_1', \dots, v_d')$ is another orthonormal basis, then
\[
\sum_{i = 1}^d |Tv_i|_{h_0}^2 = \sum_{i = 1}^d \sum_{j = 1}^d |(v_i, v_j')_{g_0}|^2 |Tv_j'|_{h_0}^2 = \sum_{j = 1}^d |Tv_j'|_{h_0}^2 \sum_{i = 1}^d |(v_i, v_j')_{g_0}|^2 = \sum_{j = 1}^d |Tv_j'|_{h_0}^2.
\]
Hence, the Frobenius norm is well-defined. If $V = W$ and $g_0 = h_0$, then we simply write $|T|_{g_0}$. If $\mathcal{T}, \mathcal{S} \subset \mathcal{L}(V, W)$, then we define their distance by
\[
\dist_{g_0, h_0}(\mathcal{T}, \mathcal{S}) := \inf \{|T - S|_{g_0, h_0} : T \in \mathcal{T}, S \in \mathcal{S}\}.
\]
If $V_1, V_2$ are orthogonal subspaces of $V$ and $V_1 + V_2 = V$, then we write $V = V_1 \oplus V_2$. In this case, $V_2$ is the orthogonal complement of $V_1$, written as $V_1^\perp = V_2$.

We denote the standard basis vectors in $\mathbb{R}^d$ by $(e_1, \dots, e_d)$. The Euclidean metric on $\mathbb{R}^d$ is denoted by $\e_d$. In the notation introduced above, we suppress the metric when all relevant vector spaces are Euclidean and write $(v, v')$, $|T|$, $\dist(\mathcal{T}, \mathcal{S})$ etc. Let $g_0$ and $g_0'$ be constant metrics on $\mathbb{R}^d$. We define their distance by
\[
|g_0 - g_0'| := \left(\sum_{i, j = 1}^d |(e_i, e_j)_{g_0} - (e_i, e_j)_{g_0'}|^2\right)^\frac{1}{2}.
\]
If $h_0$ is a constant metric on $\mathbb{R}^n$, then we use the short-hand notations $\Ort(g_0, h_0)$ and $\SO(g_0, h_0)$ instead of $\Ort((\mathbb{R}^d, g_0), (\mathbb{R}^n, h_0))$ and $\SO((\mathbb{R}^d, g_0), (\mathbb{R}^n, h_0))$. When $g_0 = h_0$, the notation is further shortened to $\Ort(g_0)$ and $\SO(g_0)$.

Let $(M, g)$ and $(N, h)$ be Riemannian manifolds. We denote the differential of a smooth map $f : M \rightarrow N$ by $df : TM \rightarrow TN$. If $N = \mathbb{R}^n$, then we can identify $T\mathbb{R}^n$ with $\mathbb{R}^{2n}$. Hence, for every $x \in M$, there exists a map $Df(x) : T_xM \rightarrow \mathbb{R}^d$ such that
\[
df_x(v) = (f(x), Df(x)v).
\]
We call $Df : TM \rightarrow \mathbb{R}^n$ the total derivative of $f$. When $f$ is only weakly differentiable, the notations $df$ and $Df$ are used in analogy with the smooth case. We denote the volume form on $M$ by $\vol_g$.

\subsection{Sobolev immersions} \label{sec:sobolev_immersions}

The following definitions are based on \cite[Sections 2.4, 2.5]{AKM}. Let $(M, g)$ be a $d$-dimensional oriented Riemannian manifold and let $(N, h)$ be a $(d + 1)$-dimensional oriented compact Riemannian manifold. Let $1 < p < \infty$. By Nash's embedding theorem, there exists a smooth isometric embedding $\iota : N \rightarrow \mathbb{R}^D$ for some $D \ge d + 1$. We write $u \in W^{1, p}(M; N)$ if $\iota \circ u \in W^{1, p}(M; \mathbb{R}^D)$ and we set $\bar{u} := \iota \circ u$. The compactness of $N$ ensures that the definition of $W^{1, p}(M; N)$ is independent of the embedding $\iota$. It is not difficult to show that for a.e. $x \in M$, $d\bar{u}_x(T_xM)$ is contained in $d\iota_{u(x)}(T_{u(x)}N)$. This allows us to define a weak differential $du_x : T_xM \rightarrow T_{u(x)}N$ for a.e. $x \in M$ by setting $du_x := (d\iota_{u(x)})^{-1} \circ d\bar{u}_x$. The definition of the weak differential does not depend on $\iota$. Throughout the paper, we do not identify Sobolev functions that agree almost everywhere. It is possible to define manifold-valued Sobolev maps intrinsically as well (see \cite{CVS}). This is the approach we take in \cite{B2}.

If $u \in W^{1, p}(M; N)$ and $\rank du_x = d$, then there exists a unique vector $\nu_u(x) \in T_{u(x)}N$ satisfying the following conditions:
\begin{enumerate}
    \item $|\nu_u(x)|_h = 1$,
    \item $(\nu_u(x), du_x(v))_h = 0$ for all $v \in T_xM$,
    \item if $(v_1, \dots, v_d)$ is a positively oriented basis of $T_xM$, then $(du_x(v_1), \dots, du_x(v_d), \nu_u(x))$ is a positively oriented basis of $T_{u(x)}N$.
\end{enumerate}
We call $\nu_u(x)$ the oriented unit normal at $x$. If $\rank du_x < d$, we set $\nu_u(x) := 0 \in T_{u(x)}N$. We define $\bar{\nu}_u(x) := D\iota(u(x))(\nu_u(x))$. The space of Sobolev immersions is given by
\[
\Imm_p(M; N) := \{u \in W^{1, p}(M; N) : \rank du_x = d \text{ for a.e. } x \in M, \, \bar{\nu}_u \in W^{1, p}(M; \mathbb{R}^D)\}.
\]
Given $u \in \Imm_p(M; N)$, we let $P_u(x)$ be the orthogonal projection of $\mathbb{R}^D$ onto the subspace that is canonically isomorphic to $T_{\bar{u}(x)}\iota(N)$. Differentiating the identity $(\bar{\nu}_u(x), \bar{\nu}_u(x)) = 1$ shows that $\bar{\nu}_u(x)$ is orthogonal to $D\bar{\nu}_u(x)(T_xM)$ for a.e. $x \in M$. Thus, $P_u(x) \circ D\bar{\nu}_u(x)$ maps $T_xM$ into $D\bar{u}(x)(T_xM)$ for a.e. $x \in M$. Consequently, there exists a unique linear map $S_u(x) : T_xM \rightarrow T_xM$ satisfying
\begin{equation} \label{eq:compact_shape_operator}
    D\bar{u}(x) \circ S_u(x) = P_u(x) \circ D\bar{\nu}_u(x).
\end{equation}
We call $S_u : TM \rightarrow TM$ the induced shape operator.

If we view $u \in \Imm_p(M; N)$ as a deformation of $M$, then we can measure its stretching and bending energies with the following functionals:
\begin{align}
    E_s(u) &:= \int_M \dist_{g, h}^p(du_x, \Ort((T_xM, g_x), (T_{u(x)}N, h_{u(x)}))) \, d\vol_g(x), \\
    E_b(u) &:= \int_M |du_x \circ S_u(x)|_{g, h}^p \, d\vol_g(x) = \int_M |P_u(x) \circ D\bar{\nu}_u(x)|_{g, \e_D}^p \, d\vol_g(x).
\end{align}
The integrand of $E_s(u)$ quantifies how far $du_x$ is from being length preserving, whereas the integrand of $E_b(u)$ measures how much the deformation of $M$ curves within the ambient space $N$. If $M$ has a pre-assigned ``shape", then we can also measure the deviation from the original shape after the deformation. To make this precise, we assume that $M$ is equipped with a smooth symmetric 2-tensor field $b$. For each $x \in M$ we let $S(x) : T_xM \rightarrow T_xM$ be the unique linear map defined by
\begin{equation} \label{eq:reference_shape_operator}
    (S(x)v, w)_{g_x} = b(v, w) \quad \text{for all } v, w \in T_xM.
\end{equation}
We call the resulting map $S : TM \rightarrow TM$ the reference shape operator. The modified bending energy is given by
\begin{equation}
    E_b^S(u) := \int_M |du_x \circ (S_u(x) - S(x))|_{g, h}^p \, d\vol_g(x).
\end{equation}
Clearly, $E_b(u) = E_b^S(u)$ if $S \equiv 0$. Finally, we also introduce the quantity
\begin{equation} \label{eq:excess_energy}
    \mathcal{E}(u) := E_b(u) + \int_M |du_x|_{g, h}^p \, d\vol_g(x),
\end{equation}
which will appear frequently in our estimates.

\section{Local quantitative rigidity estimate} \label{sec:local_rigidity}

We fix an open and bounded cube $Q \subset \mathbb{R}^d$ endowed with a metric $g$ such that
\begin{equation} \label{eq:comparable_metrics}
    \frac{1}{\lambda} \e_d \le g \le \lambda \e_d \quad \text{in } Q
\end{equation}
for some $\lambda > 0$. Let $(N, h)$ be an oriented compact manifold of dimension $d + 1$ and fix a smooth isometric embedding $\iota : N \rightarrow \mathbb{R}^D$. We also fix $p \in (1, \infty)$. In this section, we prove the following quantitative rigidity result for functions in $\Imm_p(Q; N)$. 

\begin{theorem} \label{thm:codimension1_compact_rigidity}
    Let $u \in \Imm_p(Q; N)$. There exist $x_0 \in Q$ and $R \in \Ort(g_{x_0}, \e_D)$ such that
    \[
    \int_Q |D\bar{u} - R|_{g, \e_D}^p \, dx \le C \left(|Q| \left(\osc_Q g\right)^p + E_s(u) + \diam^p(Q) \mathcal{E}(u)\right).
    \]
    The constant $C$ depends only on $d$, $p$, $\lambda$, $N$ and $\iota$.
\end{theorem}

Our proof is based on the celebrated Euclidean rigidity estimate due to Friesecke, James and Müller \cite{FJM} (see \cite[Section 2.4]{CS} for $p \neq 2$).

\begin{theorem} \label{thm:euclidean_rigidity}
    Let $U \subset \mathbb{R}^d$ be a Lipschitz domain and let $u \in W^{1, p}(U; \mathbb{R}^d)$. Then there exists $R \in \SO(\e_d)$ such that
    \begin{equation} \label{eq:friesecke_james_müller}
        \int_U |Du - R|^p \, dx \le C_U \int_U \dist^p(Du, \SO(\e_d)) \, dx
    \end{equation}
    with $C_U$ depending only on $p$ and $U$.
\end{theorem}

\begin{remark} \label{rmk:bilipschitz_deformation}
    It is easy to see that the rigidity constant $C_U$ is invariant under scaling and translations. Furthermore, if $T : U \rightarrow V \subset \mathbb{R}^d$ is a bi-Lipschitz map, then $C_V$ can be estimated in terms of $C_U$ and the Lipschitz constants of $T$ and $T^{-1}$ \cite[Theorem 2.3]{LP}.
\end{remark}

We would like to generalize Theorem \ref{thm:euclidean_rigidity} to domains endowed with non-Euclidean metrics. In order to do this, we need to add pointwise dependence to the set of rotations in the bounding term, since the rotations depend on the metric. We shall do this with a simple error estimate.

\begin{lemma} \label{lm:distance_between_sets_of_oriented_matrices}
    Let $x, y \in Q$. Then
    \[
    \dist(\SO(g_x, \e_d), \SO(g_y, \e_d)) \le \frac{\sqrt{\lambda}}{2}|g_x - g_y|.
    \]
\end{lemma}

\begin{proof}
    Define the matrices $G_x$ and $G_y$ by $(G_x)_{ij} := (e_i, e_j)_{g_x}$, $(G_y)_{ij} := (e_i, e_j)_{g_y}$. Let $R_x$ and $R_y$ be the positive square roots of $G_x$ and $G_y$, respectively. An easy computation reveals that $R_x \in \SO(g_x, \e_d)$ and $R_y \in \SO(g_y, \e_d)$. Hence,
    \[
    \dist(\SO(g_x, \e_d), \SO(g_y, \e_d)) \le |R_x - R_y|.
    \]
    We bound $|R_x - R_y|$ by $|R_x^2 - R_y^2| = |G_x - G_y|$. Let $(v_1, \dots, v_d)$ be a Euclidean orthonormal basis consisting of eigenvectors of $R_x - R_y$ with $(R_x - R_y)v_i = \mu_i v_i$. Then
    \[
    ((G_x - G_y)v_i, v_i) = (R_x(R_x - R_y)v_i, v_i) + ((R_x - R_y)R_yv_i, v_i) = \mu_i\left((R_xv_i, v_i) + (R_yv_i, v_i)\right) \ge \frac{2\mu_i}{\sqrt{\lambda}},
    \]
    where we used \eqref{eq:comparable_metrics} in the final inequality. Since $|\mu_i| = |(R_x - R_y)v_i|$, we obtain
    \[
    |G_x - G_y|^2 = \sum_{i = 1}^d |(G_x - G_y)v_i|^2 \ge \sum_{i = 1}^d ((G_x - G_y)v_i, v_i)^2 \ge \frac{4}{\lambda} \sum_{i = 1}^d |(R_x - R_y)v_i|^2 = \frac{4}{\lambda} |R_x - R_y|^2.
    \]
\end{proof}

We also require the following lemma on the uniform equivalence of the Frobenius norms with respect to $g$ and $\e_d$ in $Q$.

\begin{lemma} \label{lm:frobenius_norm_comparison}
    Let $W$ be a vector space endowed with a constant metric $h_0$. If $T \in \mathcal{L}(\mathbb{R}^d, W)$, then
    \begin{equation} \label{eq:frobenius_norm_comparison}
         \frac{1}{\sqrt{\lambda}} |T|_{g_x, h_0} \le |T|_{\e_d, h_0} \le \sqrt{\lambda} |T|_{g_x, h_0} \quad \text{for all } x \in Q.
    \end{equation}
\end{lemma}

\begin{proof}
    Fix $x \in Q$. Define the matrix $G$ by $G_{ij} := (e_i, e_j)_{g_x}$. Let $(v_1, \dots, v_d)$ be a Euclidean orthonormal basis consisting of eigenvectors of $G$. Set $\lambda_i = (v_i, v_i)_{g_x}$ and $\tilde{v}_i := v_i/\sqrt{\lambda_i}$. Then $(\tilde{v}_1, \dots, \tilde{v}_d)$ is an orthonormal basis relative to $g_x$. Consequently,
    \[
    |T|_{\e_d, h_0}^2 = \sum_{i = 1}^d |Tv_i|_{h_0}^2 = \sum_{i = 1}^d \lambda_i |T \tilde{v}_i|_{h_0}^2.
    \]
    By \eqref{eq:comparable_metrics}, we know that $1/\lambda \le \lambda_i \le \lambda$ for all $i$. Since $|T|_{g_x, h_0}^2 = \sum_{i = 1}^d |T\tilde{v}_i|_{h_0}^2$, the claim follows immediately.
\end{proof}

Compare the following result with \cite[Theorem 2.3]{LP}.

\begin{theorem} \label{thm:noneuclidean_rigidity}
    For every $u \in W^{1, p}(Q; \mathbb{R}^d)$ and $x_0 \in Q$, there exists $R \in \SO(g_{x_0}, \e_d)$ such that
    \[
    \int_Q |Du - R|_{g, \e_d}^p \, dx \le C\left(|Q|\left(\osc_Q g\right)^p + \int_Q \dist_{g, \e_d}^p(Du, \SO(g, \e_d)) \, dx\right).
    \]
    The constant $C$ depends only on $p$, $d$ and $\lambda$.
\end{theorem}

\begin{proof}
    Let $x_0 \in Q$ and $u \in W^{1, p}(Q; \mathbb{R}^d)$. If $T \in \SO(g_{x_0}, \e_d)$, then $v := u \circ T^{-1} \in W^{1, p}(T(Q); \mathbb{R}^d)$. By Theorem \ref{thm:euclidean_rigidity}, there exists $\bar{R} \in \SO(\e_d)$ such that
    \begin{equation} \label{eq:pull_back_to_euclidean}
        \int_{T(Q)} |Dv - \bar{R}|^p \, dx \le C_{T(Q)} \int_{T(Q)} \dist^p(Dv, \SO(\e_d)) \, dx.
    \end{equation}
    Due to Remark \ref{rmk:bilipschitz_deformation}, we can replace $C_{T(Q)}$ by another constant $C$ depending only on $p$, $d$ and $\lambda$, since the Lipschitz constants of $T$ and $T^{-1}$ can be bounded in terms of $\lambda$. Using the equalities 
    \begin{align*}
    \int_Q |Du - \bar{R} T|_{g_{x_0}, \e_d}^p \, dx &= \int_Q |Du \, T^{-1} - \bar{R}|^p \, dx, \\
    \int_Q \dist_{g_{x_0}, \e_d}^p(Du, \SO(g_{x_0}, \e_d)) \, dx &= \int_Q \dist^p(Du \, T^{-1}, \SO(\e_d)) \, dx
    \end{align*}
    and changing variables in \eqref{eq:pull_back_to_euclidean}, we obtain
    \[
    \int_Q |Du - \bar{R} T|_{g_{x_0}, \e_d}^p \, dx \le C \int_Q \dist_{g_{x_0}, \e_d}^p(Du, \SO(g_{x_0}, \e_d)) \, dx.
    \]
    Therefore, by Lemmas \ref{lm:distance_between_sets_of_oriented_matrices} and \ref{lm:frobenius_norm_comparison}, we have
    \begin{multline*}
        \int_Q |Du - \bar{R} T|_{g, \e_d}^p \, dx \le C \int_Q |Du - \bar{R} T|_{g_{x_0}, \e_d}^p \, dx \le C \int_Q \dist_{g_{x_0}, \e_d}^p(Du, \SO(g_{x_0}, \e_d)) \, dx \\
        \le C \left(|Q|\left(\osc_Q g\right)^p + \int_Q \dist_{g, \e_d}^p(Du, \SO(g, \e_d)) \, dx\right).
    \end{multline*}
    Since $\bar{R}T \in \SO(g_{x_0}, \e_d)$, we are done.
\end{proof}

Given $u \in \Imm_p(Q; N)$, we can compose $\bar{u}$ with a projection map onto a $d$-dimensional subspace of $\mathbb{R}^D$. Then the composition is a map between $d$-dimensional domains, to which we can apply the rigidity estimate in Theorem \ref{thm:noneuclidean_rigidity}. In order to estimate the error coming from the projection, we need to bound the $L^p$-norm of $D\bar{\nu}_u$. We do this in the next proposition in terms of the bending energy of $u$.

\begin{proposition} \label{pr:derivative_normal_estimate}
    Let $u \in \Imm_p(Q; N)$. Then for all $j = 1, \dots, d$,
    \begin{equation} \label{eq:derivative_of_the_normal_bound1}
        |\partial_{x_j} \bar{\nu}_u|^2 \le |P_u(\partial_{x_j} \bar{\nu}_u)|^2 + C |\partial_{x_j} \bar{u}|^2 \quad \text{a.e. in } Q.
    \end{equation}
    In particular,
    \[
    \int_Q |D\bar{\nu}_u|_{g, \e_D}^p \, dx \le C \mathcal{E}(u).
    \]
    The constant $C$ depends only on $d$, $p$, $\lambda$, $N$ and $\iota$.
\end{proposition}

\begin{proof}
    Set $r := D - d - 1$. If $q \in \iota(N)$, then $T_q\iota(N)$ is canonically isomorphic to a subspace of $\mathbb{R}^D$ which we denote by $V_q$. For each $q \in \iota(N)$, we can find an open neighborhood $U \subset \mathbb{R}^D$ and vector fields $n_i \in C^\infty(U; \mathbb{R}^D)$ for $i = 1, \dots, r$ such that $(n_1(q'), \dots, n_r(q'))$ is an orthonormal basis of $V_{q'}^\perp$ for all $q' \in U \cap \iota(N)$, and $|Dn_i|$ is uniformly bounded in $U$ for all $i$. 

    We choose representatives for $\bar{u}$ and $\bar{\nu}_u$ that are absolutely continuous on almost all line segments in $Q$ parallel to the coordinate axes. We fix $q \in \iota(N)$ and an integer $j \in \{1, \dots, d\}$. Assume $\bar{u}$ and $\bar{\nu}_u$ are absolutely continuous on $\ell := Q \cap \{x + te_j : t \in \mathbb{R}\}$ for some $x \in \mathbb{R}^D$. Then $I := \ell \cap \bar{u}^{-1}(U)$ is relatively open in $\ell$. Hence, we can differentiate the relation $(\bar{\nu}_u, n_i \circ \bar{u}) = 0$ in $x_j$ to get
    \[
    (\partial_{x_j}\bar{\nu}_u, n_i \circ \bar{u}) = -(\bar{\nu}_u, \partial_{x_j}(n_i \circ \bar{u})) = -(\bar{\nu}_u, Dn_i \partial_{x_j}\bar{u}) \quad \mathcal{H}^1\text{-a.e. in } I.
    \]
    Recall that $P_u(x) \in \mathcal{L}(\mathbb{R}^D; \mathbb{R}^D)$ is the orthogonal projection onto $V_{\bar{u}(x)}$. The orthogonal decomposition of $\partial_{x_j} \bar{\nu}_u$ yields
    \begin{align*}
        |\partial_{x_j} \bar{\nu}_u|^2 &= |P_u(\partial_{x_j} \bar{\nu}_u)|^2 + \sum_{i = 1}^r |(\partial_{x_j} \bar{\nu}_u, n_i \circ \bar{u} )|^2 \\
        &\le |P_u(\partial_{x_j} \bar{\nu}_u)|^2 + \sum_{i = 1}^r |Dn_i \partial_{x_j}\bar{u}|^2 \le |P_u(\partial_{x_j} \bar{\nu}_u)|^2 + C |\partial_{x_j} \bar{u}|^2 \quad \mathcal{H}^1\text{-a.e. in } I,
    \end{align*}
    where $C$ is a uniform bound for $|Dn_i|$ in $U$. Since $\ell$ was an arbitrary line segment, we conclude, using Fubini's theorem for measurable sets, that
    \begin{equation} \label{eq:derivative_of_the_normal_local_bound}
        |\partial_{x_j} \bar{\nu}_u|^2 \le |P_u(\partial_{x_j} \bar{\nu}_u)|^2 + C |\partial_{x_j} \bar{u}|^2 \quad \text{a.e. in } u^{-1}(U). 
    \end{equation}
    We can cover $\iota(N)$ with finitely many open sets $U_1, \dots, U_m$ such that the estimate \eqref{eq:derivative_of_the_normal_local_bound} holds on each $u^{-1}(U_j)$ with a uniform constant depending only on $N$ and $\iota$. Thus,
    \begin{equation} \label{eq:derivative_of_the_normal_bound2}
        |\partial_{x_j} \bar{\nu}_u|^2 \le |P_u(\partial_{x_j} \bar{\nu}_u)|^2 + C |\partial_{x_j} \bar{u}|^2 \quad \text{a.e. in } Q.
    \end{equation}
    Finally, we sum \eqref{eq:derivative_of_the_normal_bound2} over all $k$ and integrate in $Q$. We need the following lemma to conclude the proof.
\end{proof}

Since $g$ is comparable to the Euclidean metric, integrals with respect to the Riemannian volume form $\vol_g$ can be bounded by those with respect to the Lebesgue measure and vice versa. The bounding factors depend only on $\lambda$ and $d$. The result is summarized in the next lemma, which we shall use implicitly most of the time. We omit its proof, since it is elementary.

\begin{lemma} \label{lm:change_in_volume}
    If $f : Q \rightarrow [0, \infty]$ is measurable, then
    \[
    \frac{1}{\lambda^{d/2}} \int_Q f \, dx \le \int_Q f \, d\vol_g \le \lambda^{d/2} \int_Q f \, dx.
    \]
\end{lemma}

We introduce some definitions. Let $\Pi \subset \mathbb{R}^D$ be a $d$-dimensional oriented subspace. A natural orientation on $\Pi^\perp$ is defined by declaring an ordered basis $(w_1, \dots, w_{D - d})$ of $\Pi^\perp$ to be positively oriented, if $(v_1, \dots, v_d, w_1, \dots, w_{D - d})$ is a positively oriented basis of $\mathbb{R}^D$ for any positively oriented basis $(v_1, \dots, v_d)$ of $\Pi$. If $\Pi' \subset \mathbb{R}^d$ is another $d$-dimensional oriented subspace, then the distance between $\Pi$ and $\Pi'$ is defined by
\[
|\Pi - \Pi'| := \inf \left(\sum_{i = 1}^d |v_i - v_i'|^2\right)^\frac{1}{2},
\]
where the infimum ranges over all positively oriented bases $(v_1, \dots, v_d)$ and $(v_1', \dots, v_d')$ of $\Pi$ and $\Pi'$, respectively. 
The orthogonal projection from $\mathbb{R}^d$ onto $\Pi$ is denoted by $P_\Pi$.

The next proposition bounds the distance of the projection of a linear map from rotations in the projection plane by how far the map itself is from arbitrary rotations. The proof depends on two lemmas, which we state at the end.

\begin{proposition}
    Let $\Pi_0, \Pi \subset \mathbb{R}^D$ be $d$-dimensional oriented subspaces and let $T \in \mathcal{L}(\mathbb{R}^d, \Pi)$. Then
    \begin{equation} \label{eq:error_in_projection}
        |P_{\Pi_0}T - T|_{g_x, \e_D} \le |T|_{g_x, \e_D}|\Pi_0^\perp - \Pi^\perp| \quad \text{for all } x \in Q.
    \end{equation}
    If $T$ is orientation preserving, then
    \begin{equation} \label{eq:bound_for_distance_from_oriented_rotations}
        \dist_{g_x, e_D}\left(P_{\Pi_0} T, \SO\left((\mathbb{R}^d, g_x), (\Pi_0, \e_D)\right)\right) \le \dist_{g_x, e_D}(T, \Ort(g_x, \e_D)) + C|\Pi_0^\perp - \Pi^\perp| \quad \text{for all } x \in Q,
    \end{equation}
    where $C$ depends only on $d$ and $D$.
\end{proposition}

\begin{proof}
    Fix $x \in Q$. Set $r = D - d$. Let $(v_1, \dots, v_r)$ and $(w_1, \dots, w_r)$ be oriented orthonormal bases of $\Pi_0^\perp$ and $\Pi^\perp$, respectively. If $v \in \mathbb{R}^d$, then
    \[
    |P_{\Pi_0}Tv - Tv|^2 = \sum_{i = 1}^r (Tv, v_j)^2 = \sum_{i = 1}^r (Tv, v_j - w_j)^2 \le |Tv|^2 \sum_{i = 1}^r |v_j - w_j|^2.
    \]
    Taking the infimum over all positively oriented orthonormal bases of $\Pi_0^\perp$ and $\Pi^\perp$ yields $|P_{\Pi_0}Tv - Tv| \le |Tv||\Pi_0^\perp - \Pi^\perp|$ which implies $|P_{\Pi_0}T - T|_{g_x, \e_D} \le |T|_{g_x, \e_D}|\Pi_0^\perp - \Pi^\perp|$.

    To prove \eqref{eq:bound_for_distance_from_oriented_rotations}, assume first that $P_{\Pi_0}T$ is orientation preserving. By Lemma \ref{lm:in_plane_approximation}, there exists $R \in \Ort((\mathbb{R}^d, g_x), (\Pi, \e_D))$ such that $|T - R|_{g_x, \e_D} = \dist_{g_x, e_D}(T, \Ort(g_x, \e_D))$. Using Lemma \ref{lm:in_plane_approximation} again, we have
    \begin{multline*}
        \dist_{g_x, e_D}(P_{\Pi_0} T, \SO((\mathbb{R}^d, g_x), (\Pi_0, \e_D))) = \dist_{g_x, e_D}(P_{\Pi_0} T, \Ort(g_x, e_D)) \\
        \le |P_{\Pi_0} T - R|_{g_x, \e_D} \le |P_{\Pi_0} (T - R)|_{g_x, e_D} + |P_{\Pi_0} R - R|_{g_x, \e_D} \le \dist_{g_x, e_D}(T, \Ort(g_x, \e_D)) + |\Pi_0^\perp - \Pi^\perp| \sqrt{d}.
    \end{multline*}
    Lemma \ref{lm:stability_of_orientation_preservation} implies that there exists $\varepsilon > 0$, depending only on $d$ and $D$, such that $P_{\Pi_0}T$ is orientation preserving if $|\Pi_0^\perp - \Pi^\perp| < \varepsilon$. Therefore, to finish the proof, we need to handle the case $|\Pi_0^\perp - \Pi^\perp| \ge \varepsilon$. Let $R \in \Ort(g_x, \e_D)$ satisfy $|T - R|_{g_x, \e_D} = \dist_{g_x, \e_D}(T, \Ort(g_x, \e_D))$. Using the triangle inequality, it easily follows that
    \begin{multline*}
        \dist_{g_x, e_D}(P_{\Pi_0} T, \SO((\mathbb{R}^d, g_x), (\Pi_0, \e_D))) \le |P_{\Pi_0}T|_{g_x, \e_D} + \sqrt{d} \le |P_{\Pi_0}(T - R)|_{g_x, \e_D} + 2\sqrt{d} \\
        \le \dist_{g_x, e_D}(T, \Ort(g_x, \e_D)) + 2\sqrt{d} \le \dist_{g_x, e_D}(T, \Ort(g_x, \e_D)) + \frac{2\sqrt{d}}{\varepsilon} |\Pi_0^\perp - \Pi^\perp|.
    \end{multline*}
    Finally, we set $C := \sqrt{d}\max\{1, 2/\varepsilon\}$.
\end{proof}

\begin{lemma} \label{lm:in_plane_approximation}
    Let $\Pi \subset \mathbb{R}^D$ be a $d$-dimensional oriented subspace and $T \in \mathcal{L}(\mathbb{R}^d, \Pi)$. Then
    \begin{equation} \label{eq:in_plane_approximation}
        \dist_{g_x, e_D}(T, \Ort(g_x, e_D)) = \dist_{g_x, e_D}(T, \Ort((\mathbb{R}^d, g_x), (\Pi, \e_D))) \quad \text{for all } x \in Q.
    \end{equation}
    Moreover, if $T$ is orientation preserving, then
    \begin{equation} \label{eq:oriented_in_plane_approximation}
        \dist_{g_x, e_D}(T, \Ort(g_x, e_D)) = \dist_{g_x, e_D}(T, \SO((\mathbb{R}^d, g_x), (\Pi, \e_D))) \quad \text{for all } x \in Q.
    \end{equation}
\end{lemma}

\begin{proof}
    Fix $x \in Q$. To begin with, we assume $g_x = \e_d$. By the singular value decomposition, there exist $Q \in \Ort(\e_d)$, $S \in \mathcal{L}(\mathbb{R}^d)$, $R \in \Ort((\mathbb{R}^d, \e_d), (\Pi, e_D))$ and positive real numbers $\lambda_1, \dots, \lambda_d$ such that $T = RSQ$ and $Se_i = \lambda_i e_i$ for all $i$. If $L \in \Ort(\e_d, e_D)$, then
    \begin{multline*}
    |L - RS|^2 = \sum_{i = 1}^d |Le_i - RSe_i|^2 = \sum_{i = 1}^d 1 + \lambda_i^2 - 2\lambda_i(Le_i, Re_i) \\
    \ge \sum_{i = 1}^d 1 + \lambda_i^2 - 2 \lambda_i (Re_i, Re_i) = \sum_{i = 1}^d |Re_i - RSe_i|^2 = |R - RS|^2.
    \end{multline*}
    Since $LQ^{-1} \in \Ort(\e_d, e_D)$, we also have
    \[
    |L - T| = |LQ^{-1} - RS| \ge |R - RS| = |RQ - T|.
    \]
    Hence, $\dist(T, \Ort(\e_d, \e_D)) = |RQ - T|$. As $R \in \Ort((\mathbb{R}^d, \e_d), (\Pi, e_D))$, we conclude that
    \[
    \dist(T, \Ort((\mathbb{R}^d, \e_d), (\Pi, e_D))) \le |RQ - T| = \dist(T, \Ort(\e_d, \e_D)) \le \dist(T, \Ort((\mathbb{R}^d, \e_d), (\Pi, e_D))).
    \]
    Now assume $T$ is orientation preserving. Because $S$ is positive definite, it is orientation preserving as well. Therefore, $RQ$ must be orientation preserving, and we obtain
    \[
    \dist(T, \SO((\mathbb{R}^d, \e_d), (\Pi, e_D))) = \dist(T, \Ort(\e_d, \e_D)).
    \]
    If $g_x \neq \e_d$, we let $L \in \Ort(\e_d, g_x)$. Then $TL \in \Ort(\e_d, \e_D)$. We deduce from the previous step that
    \begin{multline*}
    \dist_{g_x, e_D}(T, \Ort(g_x, e_D)) = \dist(TL, \Ort(\e_d, e_D))  \\
    = \dist(TL, \Ort((\mathbb{R}^d, \e_d), (\Pi, \e_D))) = \dist_{g_x, e_D}(T, \Ort((\mathbb{R}^d, g_x), (\Pi, \e_D))).
    \end{multline*}
    The claim \eqref{eq:oriented_in_plane_approximation} follows similarly.
\end{proof}

\begin{lemma} \label{lm:stability_of_orientation_preservation}
    Let $\Pi_0, \Pi \subset \mathbb{R}^D$ be $d$-dimensional oriented subspaces. There exists $\varepsilon > 0$ depending only on $d$ and $D$ such that the restriction of $P_{\Pi_0}$ to $\Pi$ is orientation preserving if $|\Pi_0^\perp - \Pi^\perp| < \varepsilon$.
\end{lemma}

\begin{proof}
    Let $(v_{d + 1}, \dots, v_D)$ and $(w_{d + 1}, \dots, w_D)$ be positively oriented orthonormal bases of $\Pi_0^\perp$ and $\Pi^\perp$, respectively, such that $|\Pi_0^\perp - \Pi^\perp|^2 = \sum_{j = d + 1}^D |v_j - w_j|^2$. Using an orthogonal transformation, we may assume $w_j = e_j$ for $j = d + 1, \dots, D$ so that $\Pi = \lspan(e_1, \dots, e_d)$. We denote the projection map simply by $P$. Our goal is to prove that $(Pe_1, \dots, Pe_d, v_{d + 1}, \dots, v_D)$ is a positively oriented basis of $\mathbb{R}^D$ if $|\Pi_0^\perp - \Pi^\perp|$ is sufficiently small. We start with the proof of linear independence. For $i = 1, \dots, d$ we have
    \begin{equation} \label{eq:distance_from_canonical_basis}
        |e_i - Pe_i|^2 = \sum_{j = d + 1}^D |(e_i, v_j)|^2 = \sum_{j = d + 1}^D |(e_i, v_j - e_j)|^2 \le \sum_{j = d + 1}^D |v_j - e_j|^2 = |\Pi_0^\perp - \Pi^\perp|^2.
    \end{equation}
    Therefore,
    \[
    |(Pe_i, Pe_j) - \delta_{ij}| = |(Pe_i - e_i, Pe_j) + (e_i, Pe_j - e_j)| \le 2 |\Pi_0^\perp - \Pi^\perp| \quad i, j = 1, \dots, d.
    \]
    If $a_1, \dots, a_d \in \mathbb{R}$, then
    \[
    \Big|\sum_{i = 1}^d a_i Pe_i\Big|^2 = \sum_{i = 1}^d |a_i|^2 + \sum_{i,j = 1}^d a_i a_j((Pe_i, Pe_j) - \delta_{ij}).
    \]
    If $\sum_{i = 1}^d a_i Pe_i = 0$, then
    \[
    \sum_{i = 1}^d |a_i|^2 \le 2 |\Pi_0^\perp - \Pi^\perp| \Big(\sum_{i = 1}^d |a_i|\Big)^2\le 2d |\Pi_0^\perp - \Pi^\perp| \sum_{i = 1}^d |a_i|^2.
    \]
    Hence, we conclude that $a_i = 0$ for all $i$, and $(Pe_1, \dots, Pe_d)$ is linearly independent if $|\Pi_0^\perp - \Pi^\perp| < 1/2d$. We set $v_i = Pe_i$ for $i = 1, \dots, d$. It now follows that $(v_1, \dots, v_D)$ is a basis of $\mathbb{R}^D$ and \eqref{eq:distance_from_canonical_basis} implies $\sum_{i = 1}^D |v_j - e_j|^2 \le (d + 1) |\Pi_0^\perp - \Pi^\perp|^2$. Since $(e_1, \dots, e_D)$ is positively oriented, so is $(v_1, \dots, v_D)$ granted $|\Pi_0^\perp - \Pi^\perp|$ is sufficiently small.
\end{proof}

We are finally ready to prove the main theorem of this section.

\begin{proof}[Proof of Theorem \ref{thm:codimension1_compact_rigidity}]
    Roughly speaking, our strategy will be to project $\bar{u}$ to a linear space and to apply the non-Euclidean rigidity theorem \ref{thm:noneuclidean_rigidity}. In order to estimate the error due to the projection, we need to bound the variation of the tangent spaces $D\bar{u}(x)(\mathbb{R}^d)$ as $x$ ranges over $Q$.
    
    We shall denote $D\bar{u}(x)(\mathbb{R}^d)$ simply by $\Pi_x$. The subspace of $\mathbb{R}^D$ canonically isomorphic to $T_q\iota(N)$ is denoted by $V_q$. We claim that
    \begin{equation} \label{eq:normal_plane_variation}
        \int_Q \int_Q |\Pi_x^\perp - \Pi_y^\perp|^p \, dx \, dy \le C \diam^p(Q) |Q| \mathcal{E}(u).
    \end{equation}
    Since $\Pi_x^\perp = \lspan(\bar{\nu}_u(x)) \oplus V_{\bar{u}(x)}^\perp$ for a.e. $x \in Q$, we have
    \[
    |\Pi_x^\perp - \Pi_y^\perp|^2 \le |\bar{\nu}_u(x) - \bar{\nu}_u(y)|^2 + |V_{\bar{u}(x)}^\perp - V_{\bar{u}(y)}^\perp|^2 \quad \text{for a.e. } x, y \in Q.
    \]
    Set $r := D - d - 1$. Let $U \subset \mathbb{R}^D$ be a convex open set and let $n_i \in C^\infty(U; \mathbb{R}^D)$ for $i = 1, \dots, r$ such that $(n_1(q), \dots, n_r(q))$ is a positively oriented orthonormal basis of $V_{q}^\perp$ for all $q \in U \cap \iota(N)$, and $|Dn_i|$ is uniformly bounded in $U$ for all $i$. For $q, q' \in U \cap \iota(N)$, we have
    \[
    |V_q^\perp - V_{q'}^\perp|^2 \le \sum_{i = 1}^r |n_i(q) - n_i(q')|^2 \le Cr |q - q'|^2,
    \]
    where $C$ is a uniform bound for $|Dn_i|$ in $U$. Since $\iota(N)$ is compact, a covering argument yields
    \[
    |V_q^\perp - V_{q'}^\perp| \le C |q - q'| \quad \text{for all } q, q' \in \iota(N)
    \]
    with $C$ depending on $d$, $D$, $N$ and $\iota$. By Poincaré's inequality,
    \begin{equation} \label{eq:planar_variation}
        \int_Q \int_Q |V_{\bar{u}(x)}^\perp - V_{\bar{u}(y)}^\perp|^p \, dx \, dy \le C \int_Q \int_Q |\bar{u}(x) - \bar{u}(y)|^p \, dx \, dy \le C \diam^p(Q) |Q| \int_Q |D\bar{u}|^p \, dx.
    \end{equation}
    Applying Poincaré's inequality to $\bar{\nu}_u$ and using Proposition \ref{pr:derivative_normal_estimate} gives
    \begin{equation} \label{eq:normal_variation}
        \int_Q \int_Q |\bar{\nu}_u(x) - \bar{\nu}_u(y)|^p \, dx \, dy \le C \diam^p(Q) |Q| \int_Q |D\bar{\nu}_u|^p \, dx \le C \diam^p(Q) |Q| \mathcal{E}(u).
    \end{equation}
    The bound \eqref{eq:normal_plane_variation} now follows from \eqref{eq:planar_variation} and \eqref{eq:normal_variation} with the help of Lemma \ref{lm:change_in_volume}.

    By Chebyshev's inequality applied to \eqref{eq:normal_plane_variation}, there exists $x_0 \in Q$ with $\dim \Pi_{x_0} = d$ such that
    \begin{equation} \label{eq:poincare_oriented_distance}
        \int_Q |\Pi_x^\perp - \Pi_{x_0}^\perp|^p \, dx \le 2C \diam^p(Q) \mathcal{E}(u).
    \end{equation}
    Let $T \in \SO((\mathbb{R}^d, \e_d), (\Pi_{x_0}, \e_D))$ and set $v := T^{-1} \circ P_{\Pi_{x_0}} \circ \bar{u}$. Then $v \in W^{1, p}(Q; \mathbb{R}^d)$, and we can apply Theorem \ref{thm:noneuclidean_rigidity} to obtain $\bar{R} \in \SO(g_{x_0}, \e_d)$ such that
    \begin{equation} \label{eq:rigidity_applied_to_v_again}
        \int_Q |Dv - \bar{R}|_{g, \e_d}^p \, dx \le C \left(|Q| \left(\osc_Q g\right)^p + \int_Q \dist_{g, \e_d}^p(Dv, \SO(g, \e_d)) \, dx\right).
    \end{equation}
    Set $R := T\bar{R}$. Then $R \in \SO((\mathbb{R}^d, g_{x_0}), (\Pi_{x_0}, \e_D))$, and by \eqref{eq:error_in_projection}, we have
    \begin{equation} \label{eq:change_from_v_to_u}
        \begin{aligned}
            |D\bar{u}(x) - R|_{g_x, \e_D} &\le |D\bar{u}(x) - P_{\Pi_{x_0}} D\bar{u}(x)|_{g_x, \e_D} + |P_{\Pi_{x_0}} D\bar{u}(x) - R|_{g_x, \e_D} \\
            &\le |\Pi_x^\perp - \Pi_{x_0}^\perp| |D\bar{u}(x)|_{g_x, \e_D} + |Dv(x) - \bar{R}|_{g_x, \e_d} \\
            &\le \sqrt{d} |\Pi_x^\perp - \Pi_{x_0}^\perp| + 2\sqrt{D - d} \dist_{g_x, \e_D}(D\bar{u}(x), \Ort(g, \e_D)) + |Dv(x) - \bar{R}|_{g_x, \e_d}
        \end{aligned}
    \end{equation}
    for a.e. $x \in Q$. Furthermore, \eqref{eq:bound_for_distance_from_oriented_rotations} gives
    \begin{multline} \label{eq:planar_to_spacial_distance}
        \dist_{g_x, \e_d}(Dv(x), \SO(g_x, \e_d)) = \dist_{g_x, \e_D}(P_{\Pi_{x_0}} D\bar{u}(x), \SO((\mathbb{R}^d, g_x), (\Pi_{x_0}, \e_D)) \\
        \le \dist_{g_x, \e_D}(D\bar{u}(x), \Ort(g_x, \e_D)) + C|\Pi_x^\perp - \Pi_{x_0}^\perp|
    \end{multline}
    for a.e. $x \in Q$. Hence, \eqref{eq:rigidity_applied_to_v_again}, \eqref{eq:change_from_v_to_u}, \eqref{eq:planar_to_spacial_distance} and Lemma \ref{lm:change_in_volume} imply
    \[
    \int_Q |D\bar{u} - R|_{g, \e_D}^p \, dx \le C \left(|Q| \left(\osc_Q g\right)^p + E_s(u) + \int_Q |\Pi_x^\perp - \Pi_{x_0}^\perp|^p \, dx\right).
    \]
    Finally, the claim follows from \eqref{eq:poincare_oriented_distance}.
\end{proof}

\section{Asymptotic rigidity}

In this section, we prove our main result. Compare with \cite[Theorem 1.1]{AKM2}.

\begin{theorem} \label{thm:asymptotic_rigidity}
        Let $(M, g)$ and $(N, h)$ be oriented compact Riemannian manifolds of dimensions $d$ and $d + 1$, respectively. Let $\iota : N \to \mathbb{R}^D$ be a smooth isometric embedding. Assume $(u_k) \subset \Imm_p(M; N)$ satisfies
        \begin{equation} \label{eq:vanishing_and_bounded_energies}
            \lim_{k \to \infty} E_s(u_k) = 0, \quad \limsup_{k \to \infty} E_b(u_k) < \infty.
        \end{equation}
        Then there exists a subsequence $(u_{k_j})$ and $u \in \Imm_p(M; N)$ such that
        \begin{equation} \label{eq:convergence_to_local_isometry}
            \bar{u}_{k_j} \to \bar{u} \text{ in } W^{1, p}(M; \mathbb{R}^D), \quad du_x \in \Ort((T_xM, g_x), (T_{u(x)}N, h_{u(x)})) \text{ for a.e. } x \in M.
        \end{equation}
        Furthermore, if $S : TM \rightarrow TM$ is the reference shape operator with $|S|_g \in L^{p'}(M)$, where $p'$ is the Hölder conjugate of $p$, and
        \begin{equation}
            \lim_{k \to \infty} E_b^S(u_k) = 0,
        \end{equation}
        then $S = S_u$ a.e. in $M$.
\end{theorem}

Note that, in contrast to \cite[Theorem 1.1]{AKM2}, we do not prove the regularity of the limiting map. Since the main difficulty of Theorem  \ref{thm:asymptotic_rigidity} is establishing the existence of an isometric limit, we prove this as a separate lemma.

\begin{lemma} \label{lm:existence_of_the_limit}
    Let $(M, g)$, $(N, h)$ and $\iota$ be as in Theorem \ref{thm:asymptotic_rigidity}. Assume $(u_k) \subset \Imm_p(M; N)$ satisfies \eqref{eq:vanishing_and_bounded_energies}. Then there exists a subsequence $(u_{k_j})$ and $u \in W^{1, p}(M; N)$ such that
    \[
    \bar{u}_{k_j} \to \bar{u} \text{ in } W^{1, p}(M; \mathbb{R}^D), \quad du_x \in \Ort((T_xM, g_x), (T_{u(x)}N, h_{u(x)})) \text{ for a.e. } x \in M.
    \]
\end{lemma}

\begin{proof}
    Given $q \in M$, let $(U, \varphi)$ be a chart on $M$ with $q \in M$ such that $Q := \varphi(U)$ is a bounded open cube and set $v_k := u_k \circ \varphi^{-1}$. We denote the pushforward of $g$ by $\varphi$ by the same letter, and we assume without loss of generality that
    \[
    \frac{1}{\lambda} \e_d \le g \le \lambda \e_d \quad \text{in } Q
    \]
    for some $\lambda > 0$ and that $g$ is Lipschitz continuous in $Q$. We prove the claim in $U$. Since $q$ is arbitrary and $M$ is compact, the general result follows easily.

    We briefly outline the proof. In its essentials, the proof is very similar to the proof of the compactness result in \cite[Theorem 4.1]{FJM}. We use the rigidity theorem \ref{thm:codimension1_compact_rigidity} to approximate $D\bar{v}_k$ by an almost isometric piecewise constant map $G_k$. Then, using the hypotheses \eqref{eq:vanishing_and_bounded_energies}, we show that $(G_k)$ satisfies the Fréchet-Kolmogorov theorem and, therefore, prove that it has a convergent subsequence in $L^p(Q; \mathcal{L}(\mathbb{R}^d, \mathbb{R}^D))$, denoted by $G$. Finally, we show that the corresponding subsequence of $(\bar{v}_k)$ converges to a Sobolev map $\bar{v}$ with $D\bar{v} = G$. The desired limit is then given by $u := \iota^{-1} \circ \bar{v} \circ \varphi$.

    Let $l$ be the side length of $Q$. We partition $Q$ into identical open cubes with side lengths $l/t_k$, where $t_k$ is a natural number to be determined. We denote the partition by $\mathcal{P}_k$. If $Q' \in \mathcal{P}_k$, then, by Theorem \ref{thm:codimension1_compact_rigidity}, there exist $x_0 \in Q'$ and $R \in \Ort(g_{x_0}, \e_D)$ such that
    \begin{equation} \label{eq:piecewise_approximation}
        \int_{Q'} |D\bar{v}_k - R|_{g, \e_D}^p \, dx \le C \left(|Q'| \left(\osc_{Q'} g\right)^p + E_s(v_k, Q') + \diam^p(Q') \mathcal{E}(v_k, Q')\right),
    \end{equation}
    where $E_s(v_k, Q')$ and $\mathcal{E}(v_k, Q')$ denote the energies of $v_k$ restricted to $Q'$. We set $G_k(x) := R$ for all $x \in Q'$. By applying the rigidity theorem to $\bar{v}_k$ in each cube in $\mathcal{P}_k$, we obtain a piecewise constant map $G_k : Q \rightarrow \mathcal{L}(\mathbb{R}^d, \mathbb{R}^D)$.
    
    For $Q' \in \mathcal{P}_k$, we denote by $3Q'$ the concentric cube with three times the side length. We also denote the center of $Q'$ by $c_{Q'}$. Let $Q'' \in \mathcal{P}_k$ and assume $Q'' \subset 3Q' \subset Q$. We shall estimate $|G_k(c_{Q'}) - G_k(c_{Q''})|$. We apply the rigidity estimate to $\bar{v}_k$ in $3Q'$ to obtain an approximation $R'$ for $D\bar{v}_k$. Then
    \begin{multline*}
        \int_{Q'} |G_k(c_{Q''}) - R'|_{g, \e_D}^p \, dx \le C \left(\int_{Q''} |G_k(c_{Q''}) - D\bar{v}_k|_{g, \e_D}^p \, dx + \int_{3Q'} |R' - D\bar{v}_k|_{g, \e_D}^p \, dx\right) \\
        \le C\left(|3Q'| \left(\osc_{3Q'} g\right)^p + E_s(v_k, 3Q') + \diam^p(3Q') \mathcal{E}(v_k, 3Q')\right).
    \end{multline*}
    Clearly, we can also take $Q'' = Q'$. Consequently, by the triangle inequality,
    \[ 
    \int_{Q'} |G_k(c_{Q''}) - G_k(c_{Q'})|_{g, \e_D}^p \, dx \le C \left(|3Q'| \left(\osc_{3Q'} g\right)^p + E_s(v_k, 3Q') + \diam^p(3Q') \mathcal{E}(v_k, 3Q')\right).
    \]
    If $\zeta \in \mathbb{R}^d$ satisfies $\|\zeta\|_{\ell_\infty} \le l/t_k$, then the previous estimate implies more generally
    \begin{multline} \label{eq:simple_translation_vector}
        \int_{Q'} |G_k(x + \zeta) - G_k(x)|_{g, \e_D}^p \, dx \le C \sum_{Q'' \subset 3Q'} \int_{Q'} |G_k(c_{Q''}) - G_k(c_{Q'})|_{g, \e_D}^p \, dx \\
        \le C\left(|3Q'| \left(\osc_{3Q'} g\right)^p + E_s(v_k, 3Q') + \diam^p(3Q') \mathcal{E}(v_k, 3Q')\right).
    \end{multline}
    Next, we let $\zeta \in \mathbb{R}^d$ be arbitrary and set $m := \lfloor\|(t_k/l) \zeta\|_{\ell_\infty}\rfloor$. Define 
    \[
    \mathcal{G}_k(\zeta) := \{Q' \in \mathcal{P}_k : 3Q' \cup (\zeta + 3Q') \subset Q\}.
    \]
    If $Q' \in \mathcal{G}_k$, then there exist vectors $\zeta_0, \dots, \zeta_{m + 1}$ with $\|\zeta_j - \zeta_{j - 1}\|_{\ell_\infty} \le l/t_k$ such that $\zeta_0 = 0$, $\zeta_{m + 1} = \zeta$, and  $c_{Q'} + \zeta_j$ is the center of some cube $Q'_j$ in $\mathcal{P}_k$ for $j = 1, \dots, m$. Setting $Q'_0 := Q'$ and applying the triangle inequality gives
    \begin{multline} \label{eq:general_translation_vector}
        \int_{Q'} |G_k(x + \zeta) - G_k(x)|_{g, \e_D}^p \, dx \le (m + 1)^{p - 1} \sum_{j = 0}^m \int_{Q'} |G_k(x + \zeta_{j + 1}) - G_k(x + \zeta_j)|_{g, \e_D}^p \, dx \\
        \le C (m + 1)^{p - 1}\sum_{j = 0}^m \left(|3Q_j'| \left(\osc_{3Q_j'} g\right)^p + E_s(v_k, 3Q_j') + \diam^p(3Q_j') \mathcal{E}(v_k, 3Q_j')\right).
    \end{multline}
    Let $\mathcal{Q}_k(\zeta)$ be the union of all cubes in $\mathcal{G}_k(\zeta)$. We sum \eqref{eq:general_translation_vector} over all $Q' \in \mathcal{G}_k(\zeta)$ and note that, as $Q'$ ranges in $\mathcal{G}_k(\zeta)$, any cube in $\mathcal{P}_k$ appears at most $C(m + 1)$ times on the right-hand side of the inequality, where $C$ depends only on $d$. Hence,
    \begin{multline*}
        \int_{\mathcal{Q}_k(\zeta)} |G_k(x + \zeta) - G_k(x)|_{g, \e_D}^p \, dx \\
        \le C(m + 1)^p \left(|Q| \max_{Q' \in \mathcal{P}_k}\left(\osc_{3Q'} g\right)^p + E_s(v_k, Q) + \max_{Q' \in \mathcal{P}_k} \diam^p(3Q') \mathcal{E}(v_k, Q)\right).
    \end{multline*}
    We observe that
    \begin{equation} \label{eq:oscillation_and_diameter_bound}
        \osc_{3Q'} g \le CL\frac{l}{t_k}, \quad \diam(3Q') \le C \frac{l}{t_k} \quad \text{for all } Q' \in \mathcal{P}_k,
    \end{equation}
    where $C$ depends only on $d$, and $L$ is the Lipschitz constant of $g$. As a result,
    \begin{equation} \label{eq:bound_on_uk}
        \int_{\mathcal{Q}_k(\zeta)} |G_k(x + \zeta) - G_k(x)|_{g, \e_D}^p \, dx \le C \left(\|\zeta\|_{\ell_\infty} + \frac{l}{t_k}\right)^p \left(|Q| L^p + \left(\frac{t_k}{l}\right)^p E_s(v_k, Q) + \mathcal{E}(v_k, Q)\right).
    \end{equation}
    On the other hand, extending $G_k$ by 0 to $\mathbb{R}^d \setminus Q$, we see that
    \begin{equation} \label{eq:bound_outside_uk}
        \int_{\mathbb{R}^d \setminus \mathcal{Q}_k(\zeta)} |G_k(x + \zeta) - G_k(x)|_{g, \e_D}^p \, dx \le C |Q \setminus \mathcal{Q}_k(\zeta)|.
    \end{equation}
    It is clear that
    \[
    \lim_{\|\zeta\|_{\ell_\infty} \to 0} \limsup_{k \to \infty} |Q \setminus \mathcal{Q}_k(\zeta)| = 0.
    \]
    As $\varphi : (U, g) \rightarrow (Q, g)$ is an isometry, we have
    \[
    E_s(v_k, Q) = E_s(u_k, U), \quad \mathcal{E}(v_k, Q) = \mathcal{E}(u_k, U).
    \]
    Now, we choose a sequence $(t_k)$ tending to infinity such that $\limsup_{k \to \infty} t_k^p E_s(u_k, U) < \infty$. It follows from \eqref{eq:vanishing_and_bounded_energies} that
    \[
    \limsup_{k \to \infty} |Q| L^p + \left(\frac{t_k}{l}\right)^p E_s(v_k, Q) + \mathcal{E}(v_k, Q) < \infty.
    \]
    Consequently, \eqref{eq:bound_on_uk} and \eqref{eq:bound_outside_uk} imply
    \[
    \lim_{\|\zeta\|_{\ell_\infty} \to 0} \limsup_{k \to \infty} \int_{\mathbb{R}^d} |G_k(x + \zeta) - G_k(x)|_{g, \e_D}^p \, dx = 0.
    \]
    Hence, by the Fréchet-Kolmogorov theorem, there exists a subsequence $(G_{k_j})$ and a map $G : Q \rightarrow \mathcal{L}(\mathbb{R}^d, \mathbb{R}^D)$ such that
    \[
    \lim_{j \to \infty} \int_{Q} |G_{k_j} - G|_{g, \e_D}^p \, dx = 0.
    \]
    From \eqref{eq:vanishing_and_bounded_energies} and \eqref{eq:piecewise_approximation} we get $|D\bar{v}_k - G_k|_{g, \e_D} \to 0$ in $L^p(Q)$. Therefore, $|D\bar{v}_{k_j} - G|_{g, \e_D} \to 0$ in $L^p(Q)$ as $j \to \infty$. Clearly, $G(x) \in \Ort(g_x, \e_D)$ for a.e. $x \in Q$.
    
    In order to prove the convergence of the sequence $(\bar{v}_{k_j})$ we assume that the averages $(\bar{v}_{k_j})_Q$ converge in $\iota(N)$. As $\iota(N)$ is compact, there is no loss of generality. By the Poincaré inequality,
    \begin{multline*}
        \int_{Q} |\bar{v}_{k_j} - \bar{v}_{k_i}|_{g, \e_D}^p \, dx \le C\left(\int_Q |(\bar{v}_{k_j} - \bar{v}_{k_i}) - (\bar{v}_{k_j} - \bar{v}_{k_i})_Q|_{g, \e_D}^p + |(\bar{v}_{k_j} - \bar{v}_{k_i})_Q|_{g, \e_D}^p \, dx\right) \\
        \le C \left(\int_Q |D\bar{v}_{k_j} - D\bar{v}_{k_i}|_{g, \e_D}^p dx + |Q||(\bar{v}_{k_j})_Q - (\bar{v}_{k_i})_Q|^p\right).
    \end{multline*}
    Hence, $(\bar{v}_{k_j})$ is a Cauchy sequence. Let $\bar{v}$ be its limit. Clearly, $\bar{v} \in W^{1, p}(Q; \mathbb{R}^D)$ and $D\bar{v} = G$. Since $\bar{v}(x) \in \iota(N)$ for a.e. $x \in Q$, $v := \iota^{-1} \circ \bar{v}$ is well-defined. It follows that $v \in W^{1, p}(Q; N)$ and $dv_x \in \Ort((T_x\mathbb{R}^d, g_x), (T_{u(x)}N, h_{u(x)}))$ for a.e. $x \in Q$. Setting $u := v \circ \varphi$ concludes the proof.
\end{proof}

We are ready to prove Theorem \ref{thm:asymptotic_rigidity}.

\begin{proof}[Proof of Theorem \ref{thm:asymptotic_rigidity}]
    By Lemma \ref{lm:existence_of_the_limit}, there exist a subsequence of $(u_k)$, not relabeled, and a Sobolev map $u \in W^{1, p}(M; N)$ such that
    \[
    \bar{u}_k \to \bar{u} \text{ in } W^{1, p}(M; \mathbb{R}^D), \quad du_x \in \Ort((T_xM, g_x), (T_{u(x)}N, h_{u(x)})) \text{ for a.e. } x \in M.
    \]
    Since $\rank du_x = d$ a.e. in $M$, the oriented unit normal $\nu_u$ is well-defined a.e. in $M$. We show that $\bar{\nu}_u \in W^{1, p}(M; \mathbb{R}^D)$, which proves that $u \in \Imm_p(M; N)$. Without loss of generality, we assume $\bar{u}_k \to \bar{u}$ and $D\bar{u}_k \to D\bar{u}$ pointwise a.e. in $M$.

    We denote the subspace of $\mathbb{R}^D$ canonically isomorphic to $T_q\iota(N)$ by $V_q$. Set $r := D - d - 1$. Let $U \subset \mathbb{R}^D$ be open and let $n_i \in C^\infty(U; \mathbb{R}^D)$ for $i = 1, \dots, r$ such that $(n_1(q), \dots, n_r(q))$ is a positively oriented orthonormal basis of $V_{q}^\perp$ for all $q \in U \cap \iota(N)$. Assume
    \begin{equation} \label{eq:generic_convergence}
        \bar{u}_k(x) \to \bar{u}(x), \quad D\bar{u}_k(x) \to D\bar{u}(x), \quad D\bar{u}(x) \in \Ort(g_x, \e_D)
    \end{equation}
    for some $x \in \bar{u}^{-1}(U)$. Then
    \[
    \lim_{k \to \infty} (\bar{\nu}_{u_k}(x), n_i(\bar{u}(x)) - n_i(\bar{u}_k(x))) = \lim_{k \to \infty} (\bar{\nu}_{u_k}(x), n_i(\bar{u}_k(x))) = 0 \implies \lim_{k \to \infty} (\bar{\nu}_{u_k}(x), n_i(\bar{u}(x))) = 0.
    \]
    On the other hand, if $v \in \mathbb{R}^d$, then
    \[
    \lim_{k \to \infty} (\bar{\nu}_{u_k}(x), D(\bar{u} - \bar{u}_k)(x)v) = \lim_{k \to \infty} (\bar{\nu}_{u_k}(x), D\bar{u}_k(x)v) = 0 \implies \lim_{k \to \infty} (\bar{\nu}_{u_k}(x), D\bar{u}(x)v) = 0.
    \]
    Let $(v_1, \dots, v_d)$ be an orthonormal basis of $(\mathbb{R}^d, g_x)$. From $V_{\bar{u}(x)} = \lspan(\bar{\nu}_u(x)) \oplus D\bar{u}(\mathbb{R}^d)$ we obtain
    \begin{multline*}
        1 = |\bar{\nu}_{u_k}(x)|^2 = |(\bar{\nu}_{u_k}(x), \bar{\nu}_u(x))|^2 + \sum_{i = 1}^r |(\bar{\nu}_{u_k}(x), n_i(\bar{u}(x)))|^2 + \sum_{i = 1}^d |(\bar{\nu}_{u_k}(x), D\bar{u}(x)v_i)|^2 \\
        \implies \lim_{k \to \infty} |(\bar{\nu}_{u_k}(x), \bar{\nu}_u(x))| = 1.
    \end{multline*}
    Since $(D\bar{u}_k(x)v_1, \dots, D\bar{u}_k(x)v_d, \bar{\nu}_{u_k}(x))$ and $(D\bar{u}(x)v_1, \dots, D\bar{u}(x)v_d, \bar{\nu}_u(x))$ are both positively oriented and $D\bar{u}_k(x) \to D\bar{u}(x)$, we get $(\bar{\nu}_{u_k}(x), \bar{\nu}_u(x)) \to 1$. This proves that $|\bar{\nu}_{u_k}(x) - \bar{\nu}_u(x)| \to 0$. As \eqref{eq:generic_convergence} is satisfied for a.e. $x \in \bar{u}^{-1}(U)$, we conclude by the dominated convergence theorem that $\bar{\nu}_{u_k} \to  \bar{\nu}_u$ in $L^p(\bar{u}^{-1}(U))$. Convergence in $L^p(M)$ follows from the arbitrariness of $U$. Furthermore, by covering $M$ with finitely many charts and applying Proposition \ref{pr:derivative_normal_estimate} locally, we deduce
    \[
    \limsup_{k \to \infty} \int_M |D\bar{\nu}_{u_k}|_{g, \e_D}^p \, d\vol_g \le \limsup_{k \to \infty} C \mathcal{E}(u_k) < \infty.
    \]
    Hence, $\bar{\nu}_u \in W^{1, p}(M; \mathbb{R}^D)$ and $D\bar{\nu}_{u_k} \rightharpoonup D\bar{\nu}_u$ in $L^p$.

    To conclude the proof, assume $\lim_{k \to \infty} E_b^S(u_k) = 0$. Let $W \subset M$ be open and let $v_i \in C^\infty(W; TM)$ for $i = 1, \dots, d$ such that $(v_1(x), \dots, v_d(x))$ is a basis of $T_xM$, that is, $(v_1, \dots, v_d)$ is a frame in $W$. We show that $D\bar{u}(Sv_i) = D\bar{u}(S_uv_i)$ a.e. in $W$. By Hölder's inequality, we have $D\bar{u}_k \circ S \to  D\bar{u} \circ S$ in $L^1$. Thus, the hypothesis
    \[
    \lim_{k \to \infty} \int_M |D\bar{u}_k \circ (S_{u_k} - S)|_{g, \e_D}^p \, d\vol_g = 0
    \]
    yields $P_{u_k} \circ D\bar{\nu}_{u_k} = D\bar{u}_k \circ S_{u_k} \to D\bar{u} \circ S$ in $L^1$. Let $\eta \in C_c^\infty(W; \mathbb{R}^D)$ be an arbitrary vector field. Then
    \begin{equation} \label{eq:first_convergence}
        \lim_{k \to \infty} \int_W (P_{u_k}(D\bar{\nu}_{u_k}v_i), \eta) \, d\vol_g = \int_W (D\bar{u}(Sv_i), \eta) \, d\vol_g
    \end{equation}
    On the other hand,
    \[
    \int_W (P_{u_k}(D\bar{\nu}_{u_k}v_i), \eta) \, d\vol_g = \int_W (D\bar{\nu}_{u_k}v_i, P_{u_k}\eta) \, d\vol_g.
    \]
    Since $\bar{u}_k \to \bar{u}$ pointwise a.e. in $M$, $P_{u_k}\eta \to P_u\eta$ in $L^q$ by the dominated convergence theorem. Thus, the weak convergence of $D\bar{\nu}_{u_k}$ implies
    \begin{multline*}
        \lim_{k \to \infty} \int_W (P_{u_k}(D\bar{\nu}_{u_k}v_i), \eta) \, d\vol_g = \int_W (D\bar{\nu}_uv_i, P_u\eta) \, d\vol_g \\
        = \int_W (P_u(D\bar{\nu}_uv_i), \eta) \, d\vol_g = \int_W (D\bar{u}(S_uv_i), \eta) \, d\vol_g.
    \end{multline*}
    Hence, from \eqref{eq:first_convergence} and the arbitrariness of $\eta$, we obtain $D\bar{u}(Sv_i) = D\bar{u}(S_uv_i)$ a.e. in $W$. As $D\bar{u}$ is injective a.e. in $W$, we conclude that $S = S_u$ a.e. in $W$. Every point in $M$ has a neighborhood with a frame, so that $S = S_u$ holds a.e. in $M$.
\end{proof}

\section*{Acknowledgements}

This work is funded by the Deutsche Forschungsgemeinschaft (DFG, German Research Foundation) under Germany's Excellence Strategy EXC 2044–390685587, Mathematics Münster: Dynamics–Geometry–Structure. This paper is an outgrowth of my Master's thesis at the University of Bonn. I would like to thank Stefan Müller for proposing the problem and for insightful ideas that shaped the direction of this work. I would also like to thank Kerrek Stinson for helpful discussions.

\bibliography{references}
\bibliographystyle{plain}

\end{document}